\documentclass[a4paper]{amsart}
\usepackage{color}

\usepackage{amsmath,amssymb,amsfonts}
\usepackage{amsthm}
\usepackage{stmaryrd}

\newtheorem{theorem}{Theorem}[section]
\newtheorem{lemma}[theorem]{Lemma}
\newtheorem{proposition}[theorem]{Proposition}

\theoremstyle{definition}
\newtheorem{definition}[theorem]{Definition}
\newtheorem{assumption}[theorem]{Assumption}
\newtheorem{remark}[theorem]{Remark}
\begin{document}
\setlength\arraycolsep{2pt}
\title[Almost Periodic Solutions of fDE with Statistical Application]{Almost Periodic and Periodic Solutions of Differential Equations Driven by the Fractional Brownian Motion with Statistical Application}
\author{Nicolas MARIE*}
\address{*Laboratoire Modal'X, Universit\'e Paris Nanterre, Nanterre, France}
\email{nmarie@parisnanterre.fr}
\address{*ESME Sudria, Paris, France}
\email{nicolas.marie@esme.fr}
\author{Paul RAYNAUD DE FITTE$^{\dag}$}
\address{$^{\dag}$LMRS, Universit\'e de Rouen Normandie, Rouen, France}
\email{prf@univ-rouen.fr}
%\keywords{}
\date{}

% Abstract.

%
\begin{abstract}
We show that the unique solution to a semilinear stochastic
differential equation with almost periodic coefficients
driven by a fractional Brownian motion is almost periodic in a sense
related to random dynamical systems. This type of almost periodicity
allows for the construction of a 
consistent estimator of the drift
parameter in the almost periodic and periodic cases. 
\end{abstract}
\maketitle
\noindent
\tableofcontents
% \noindent
% \textbf{Acknowledgments.}
% This work was funded by RFBR and CNRS, project number PRC2767. 
% %This work was supported by the CNRS--RFBR project DIRaH.
%

% Section : Introduction.

%
\section{Introduction}\label{section_introduction}
%%%%%%%%%%%%%%%%%%%%%%%%%%%%%%%%%%%%%%%%%%%%%%%%%%%%%%%%%
Since its introduction by Harald Bohr in the 1920s, the notion of
almost periodicity has found many applications in the qualitative 
study of ordinary differential equations and dynamical systems, and
many generalisations have been proposed and applied: almost periodicity in the
sense of Stepanov, or Weyl, or Besicovitch, almost automorphy, asymptotic
almost periodicity, etc. See Andres et al.~\cite{andres-bersani} for a
survey and a comparison of some of these notions.
\\
The application of almost periodicity to stochastic
differential equations in the framework of It\^o calculus
seems to start in the 1980s
with the Romanian school,
in a series of papers by Constantin Tudor and his collaborators:
\cite{DaPrato-Tudor95,halanay87,Morozan-Tudor89,Tudor92affine,Tudor92flows},
to cite but a few.  
Each known notion of almost periodicity for deterministic functions 
forks into several possible definitions for stochastic processes,
mainly: almost periodicity in distribution (in various senses), in
probability, or in square mean,  
see the surveys by Tudor \cite{Tudor95ap_processes}
and Bedouhene et al.~\cite{BMRF}.
However, almost periodicity in probability or in square mean appeared to be
inapplicable to stochastic differential equations,  
see \cite{almostautomorphy,counterexamples}. 
Recently, a new definition of almost periodicity for stochastic
processes has been introduced
in Zhang and Zheng \cite{zhang-zheng19ap} and 
Raynaud de Fitte \cite{RF20}, namely $\theta$-almost periodicity,
where $\theta$ is the Wiener shift.
One motivation of \cite{RF20} was to circumvent the limitations of
``plain'' almost periodicity in square mean by introducing the
action of a group $\theta$ of measure preserving transformations on the
underlying probality space.
\\
This paper is devoted to $\theta$-almost periodicity (in Bohr's sense)
in square mean and statistical estimation
for solutions to stochastic differential equations driven by a
fractional Brownina motion with Hurst index greater than $1/2$.
The paper is organized as follows. 
\\
We present $\theta$-almost periodicity in Section
\ref{section_preliminaries}, along with some preliminaries on
stochastic integration with respect to fractional Brownian motion. 
\\
In Section
\ref{section_ap_solutions}, 
we prove the existence and uniqueness of a $\theta$-almost periodic
in square mean (resp.~$\theta$-periodic) solution to
\begin{equation}\label{main_equation}
dX(t) =
(AX(t) + b(t,X(t)))dt +\sigma(t)dB(t)
\textrm{ $;$ }
t\in\mathbb R
\end{equation}
where $A\in\mathcal M_d(\mathbb R)$ with $d\in\mathbb N^*$, $b
:\mathbb R\times\mathbb R^d\rightarrow\mathbb R^d$ and $\sigma
:\mathbb R\rightarrow\mathcal M_d(\mathbb R)$ are continuous
functions, $B$ is a $d$-dimensional two-sided fractional Brownian
motion (fBm) of Hurst index $H\in ]1/2,1[$, and 
the functions $t\mapsto b(t,x)$, $x\in\mathbb R^d$ and $\sigma$ are
assumed to be almost periodic (resp.~periodic).
We also show, in Remark \ref{rem:counterexample}, that
``plain'' almost periodicity in square mean is inapplicable 
to stochastic equations driven by fractional Brownian motion, 
despite some papers claiming the existence of nontrivial almost periodic
solutions in square mean. 
\\
Section \ref{section_parameter_estimation} is devoted to parametric
estimation, using some features of $\theta$-almost periodicity. 
Along the last two decades, many authors investigated statistical inference in differential equations driven by fractional Brownian motion (fDE). Most references on the estimation of the trend component in fDE deal with parametric estimators under a dissipativity condition on the drift function ensuring the existence and uniqueness of a stationary solution (see Kleptsyna and Le Breton \cite{KB01}, Tudor and Viens \cite{TV07}, Hu and Nualart \cite{HN10}, Neuenkirch and Tindel \cite{NT14}, Hu et al.~\cite{HNZ18}, etc.). Some recent papers deal with parametric estimators of the drift parameter in the fractional Langevin equation with periodic mean (see Dehling et al.~\cite{DFW17} and Bajja et al.~\cite{BEV17}). Having in mind these two research fields, for $d = 1$, Section \ref{section_parameter_estimation} deals with the convergence of a Skorokhod's integral based least-square type estimator, similar to those of Hu et al.~\cite{HNZ18}, of the parameter $\vartheta > 0$ in
\begin{equation}\label{main_equation_estimation}
dX(t) =
-\vartheta(X(t) - b_0(t,X(t)))dt +\sigma(t)dB(t)
\textrm{ $;$ }
t\in\mathbb R,
\end{equation}
where the function $b_0 :\mathbb R^2\rightarrow\mathbb R$ is
continuous and $t\mapsto b_0(t,x)$ is almost periodic for every
$x\in\mathbb R$.
As with stationarity in Hu et al.~\cite{HNZ18}, the periodicity or
almost periodicity of the
solution to Equation (\ref{main_equation_estimation}) under the
conditions of Section \ref{section_parameter_estimation} allows to
prove the consistency of the mentioned least-square type estimator of
$\vartheta$. To our knowledge, this problem has not yet been
investigated, even for periodic diffusion processes. 
\\
\\
\textbf{Notations and basic properties:}
\begin{enumerate}
 \item For every $s,t\in\mathbb R$ such that $s < t$, $\Delta_{s,t} :=\{(u,v)\in [s,t]^2 : u < v\}$.
 \item For every function $f$ from $\mathbb R$ into $\mathbb R^d$ and $(s,t)\in\mathbb R^2$, $f(s,t) := f(t) - f(s)$.
 \item Consider a real interval $I$. The vector space of continuous functions from $I$ into $\mathbb R^d$ is denoted by $C^0(I,\mathbb R^d)$ and equipped with the uniform norm $\|.\|_{\infty,I}$ defined by
 \begin{displaymath}
 \|f\|_{\infty,I} :=
 \sup_{u\in I}\|f(u)\|
 \textrm{ $;$ }
 \forall f\in C^0(I,\mathbb R^d).
 \end{displaymath}
 In particular, $\|.\|_{\infty,s,t} :=\|.\|_{\infty,[s,t]}$ for every $s,t\in\mathbb R$ such that $s < t$.
 \item Consider $s,t\in\mathbb R$ such that $s < t$. The set of all dissections of $[s,t]$ is denoted by $\mathfrak D_{[s,t]}$.
 \item Consider $n\in\mathbb N^*$. The vector space of infinitely continuously differentiable maps $f :\mathbb R^n\rightarrow\mathbb R$ such that $f$ and all its partial derivatives have polynomial growth is denoted by $C_{\mathbf p}^{\infty}(\mathbb R^n;\mathbb R)$.
 \item Consider a probability space $(\Omega,\mathcal A,\mathbb P)$. Let $\mathbb L^0(\Omega;\mathbb R^d)$ be the space of equivalence classes, for the almost everywhere equality, of measurable mappings from $\Omega$ into $\mathbb R^d$. For every $p\geqslant 1$, the usual distance on $\mathbb L^p(\Omega;\mathbb R^d)$ is denoted by $d_p$.
\end{enumerate}
%

% Section : Preliminaries.

%
\section{Preliminaries}\label{section_preliminaries}
This section provides some preliminary material on almost periodicity
and on stochastic integrals with respect to fractional Brownian
motion. 
%

% Subsection : Almost periodic functions and \theta-almost periodic proceses.

%
\subsection{Almost periodic functions and $\theta$-almost periodic processes}
This subsection deals with almost periodic functions and almost periodic processes with respect to a metric dynamical system.
%

% Definition : Almost periodic functions.

%
\begin{definition}\label{almost_periodic_functions}
\begin{enumerate}
 \item A set $A\subset\mathbb R$ is \emph{relatively dense} if, for every $\varepsilon > 0$, there exists $l > 0$ such that every interval of length $l$ has a nonempty intersection with $A$.
 \item Let $f :\mathbb R\rightarrow\mathbb R^d$ be a continuous function. For any $\varepsilon > 0$, $\tau > 0$ is an \emph{$\varepsilon$-almost period} of $f$ if
 \begin{displaymath}
 \forall t\in\mathbb R
 \textrm{, }
 \|f(t +\tau) - f(t)\|\leqslant\varepsilon.
 \end{displaymath}
 \item A continuous function $f :\mathbb R\rightarrow\mathbb R^d$ is \emph{almost periodic} (in Bohr's sense) if, for every $\varepsilon > 0$, the set of its $\varepsilon$-almost periods is relatively dense.

 \item A continuous function $f :\mathbb R\times\mathbb
   R^d\rightarrow\mathbb R^d$ is \emph{almost periodic uniformly with
     respect to compact subsets of $\mathbb R^d$} if, for every
   compact subset $K$ of $\mathbb R^d$, the map 
 \begin{displaymath}
 t\in\mathbb R\longmapsto
 f(t,.)_{|K}
 \end{displaymath}
 is almost periodic. 
\end{enumerate}
\end{definition}
\noindent
Now, let us state the mean value theorem and Parseval's equality for almost periodic functions. These results are proved in Levitan and Zhikov \cite{LZ82}, Chapter 2. The reader can also refer to Corduneanu \cite{corduneanu}.
%

% Proposition : The mean value theorem for almost periodic function.

%
\begin{proposition}\label{mean_value_ap_functions}
For every almost periodic function $f :\mathbb R\rightarrow\mathbb C$, its mean value
\begin{displaymath}
\mathcal M(f) :=
\lim_{t\rightarrow\infty}
\frac{1}{t}\int_{0}^{t}f(s)ds
\end{displaymath}
exists.
\end{proposition}
%

% Proposition : Parseval's inequality for almost periodic functions.

%
\begin{proposition}\label{Parseval_equality_ap_functions}
For every almost periodic function $f :\mathbb R\rightarrow\mathbb C$, its spectrum
\begin{displaymath}
\mathbb S(f) :=
\{\lambda\in\mathbb R :\mathcal M(fe^{i\lambda .})\not= 0\}
\end{displaymath}
is at least countable and, for every sequence $(\lambda_n)_{n\in\mathbb N}$ of elements of $\mathbb S(f)$,
\begin{displaymath}
\sum_{n = 1}^{\infty}
|\mathcal M(fe^{i\lambda_n .})|^2 =
\mathcal M(|f|^2).
\end{displaymath}
\end{proposition}
\noindent
Let $(\Omega,\mathcal A,\mathbb{P},\theta)$ be a metric dynamical
system (in the sense of Arnold \cite{ARNOLD98}), that is,
$(\Omega,\mathcal{A},\mathbb{P})$ is a probability space, and
$\theta=(\theta_t)_{t\in\mathbb{R}}$ is a group of measure preserving
transformations on $\Omega$, that is, each 
$\theta_t :\,\Omega\rightarrow\Omega$ is $\mathcal{A}$-measurable,
with 
$\mathbb{P}(\theta_t^{-1}(A)) =\mathbb{P}(A)$ for all $A\in\mathcal{A}$, and
%$\theta_0=\textrm{Id}_\Omega$,
$\theta_{s+t}=\theta_s\circ\theta_t$
for all $s,t\in\mathbb{R}$.  
%

% Definition : Translation.

%
\begin{definition}\label{translation}
The \emph{translation} of a continuous process $Y$ is the $C^0(\mathbb
R,\mathbb L^0(\Omega;\mathbb R^d))$-valued map $\mathfrak TY$ defined
by 
\begin{displaymath}
\mathfrak T_{\tau}Y(t,\omega) :=
Y(t +\tau,\theta_{-\tau}\omega)
\end{displaymath}
for every $\omega\in\Omega$ and $t,\tau\in\mathbb R$.
\end{definition}
%

% Definition : Almost periodic processes.

%
\begin{definition}\label{almost periodic_processes}
Let $Y$ be a continuous process such that $Y(t)\in\mathbb L^p(\Omega;\mathbb R^d)$ for every $t\in\mathbb R$.
\begin{enumerate}
 \item For any $\varepsilon > 0$, $\tau > 0$ is a
   \emph{$\theta$-$\varepsilon$-almost period in $p$-mean} of $Y$ if 
 \begin{displaymath}
 \sup_{t\in\mathbb R}
 d_p(\mathfrak T_{\tau}Y(t),Y(t))\leqslant\varepsilon.
 \end{displaymath}
 \item The continuous process $Y$ is \emph{$\theta$-almost periodic in $p$-mean} if $(t,\tau)\mapsto\mathfrak T_{\tau}Y(t)$ is continuous for the distance $d_p$ and, for every $\varepsilon > 0$, the set of its $\theta$-$\varepsilon$-almost periods is relatively dense.
 \item The continuous process $Y$ is \emph{$\theta$-$\tau$-periodic}
   with $\tau > 0$ if $\mathfrak T_{\tau}Y = Y$. 
\end{enumerate}
\end{definition}
\noindent
The following proposition provides a compactness result which is crucial in the first step of the proof of Proposition \ref{almost_periodic_solutions}.
%

% Proposition : Compactness property of almost periodic processes.

%
\begin{proposition}\label{compactness_ap_processes}
Consider a continuous process $Y$ and a compact interval $J\subset\mathbb R$. Assume that $Y$ is $\theta$-almost periodic in $p$-mean. Then,
\begin{enumerate}
 \item The set $\{\mathfrak T_{\tau}Y(t)\textrm{ $;$ }t\in J\textrm{, }\tau\in\mathbb R\}$ is relatively compact in $\mathbb L^p(\Omega;\mathbb R^d)$.
 \item For every $\varepsilon > 0$, there exists a compact subset $K$ of $\mathbb R^d$ such that
 \begin{displaymath}
 \sup_{t\in\mathbb R}\mathbb P(Y(t)\not\in K)\leqslant\varepsilon.
 \end{displaymath}
\end{enumerate}
\end{proposition}
\noindent
See \cite[Proposition 3.10 and Subsection 3.3]{RF20} for a proof.
%

% Subsection : Wiener and Skorokhod integrals with respect to the fBm.

%
\subsection{Wiener and Skorokhod integrals with respect to the fBm}
This subsection deals with the definitions and basic properties of Wiener's integral and of Skorokhod's integral with respect to the fractional Brownian motion of Hurst index greater than $1/2$.
%

% Definition : Riemann's sums.

%
\begin{definition}\label{Riemann_sum}
Let $y$ (resp. $w$) be a continuous function from $\mathbb R$ into $\mathcal M_d(\mathbb R)$ (resp. $\mathbb R^d$). Consider a dissection $D = (t_0,\dots,t_m)$ of $[s,t]$ with $m\in\mathbb N^*$ and $s,t\in\mathbb R$ such that $s < t$. The \emph{Riemann sum} of $y$ with respect to $w$ on $[s,t]$ for the dissection $D$ is
\begin{displaymath}
J_{y,w,D}(s,t) :=
\sum_{k = 0}^{m - 1}y(t_k)(w(t_{k + 1}) - w(t_k)).
\end{displaymath}
\end{definition}
\noindent
\textbf{Notation.} With the notations of Definition \ref{Riemann_sum}, the mesh of the dissection $D$ is
\begin{displaymath}
\delta(D) :=
\max_{k\in\llbracket 0,m - 1\rrbracket}
|t_{k + 1} - t_k|.
\end{displaymath}
\noindent
In the sequel, $(\Omega,\mathcal A,\mathbb P)$ is the canonical probability space associated to the $d$-dimensional fractional Brownian motion $B = (B_1,\dots,B_d)$.
\\
\\
On the one hand, consider the Banach space
\begin{displaymath}
|\mathcal H| :=
\{h\in\mathbb L^0(\mathbb R) :
\|h\|_{|\mathcal H|} <\infty\},
\end{displaymath}
where $\|.\|_{|\mathcal H|}$ is the norm defined by
\begin{displaymath}
\|h\|_{|\mathcal H|} :=
H(2H - 1)
\int_{-\infty}^{\infty}\int_{-\infty}^{\infty}
|t - s|^{2H - 2}|h(s)|\cdot |h(t)|dsdt
\end{displaymath}
for every $h\in\mathbb L^0(\mathbb R)$ (see Pipiras and Taqqu \cite{PT00}, Section 4).
%

% Theorem : Wiener's integral.

%
\begin{theorem}\label{Wiener_integral}
Consider $s,t\in\mathbb R$ such that $s < t$, $j\in\{1,\dots,d\}$ and $h\in\mathbb L^0(\mathbb R)$ such that $h\mathbf 1_{[s,t]}\in|\mathcal H|$. There exists a unique $J_{h,B_j}(s,t)\in\mathbb L^2(\Omega;\mathbb R)$ such that for every sequence $(D_n)_{n\in\mathbb N}$ of dissections of $[s,t]$ satisfying $\delta(D_n)\rightarrow 0$ as $n\rightarrow\infty$,
\begin{displaymath}
\lim_{n\rightarrow\infty}
\mathbb E(|J_{h,B_j}(s,t) - J_{h,B_j,D_n}(s,t)|^2) = 0
\end{displaymath}
and
\begin{equation}\label{Wiener_integral_1}
\mathbb E(J_{h,B_j}(s,t)^2) =
H(2H - 1)
\int_{s}^{t}\int_{s}^{t}h(u)h(v)|v - u|^{2H - 2}dudv.
\end{equation}
The random variable $J_{h,B_j}(s,t)$ is the Wiener integral of $h$ with respect to $B_j$ on $[s,t]$ and it is denoted by
\begin{displaymath}
\int_{s}^{t}h(u)dB_j(u).
\end{displaymath}
\end{theorem}
\noindent
(See Huang and Cambanis \cite{HC78}, Section 1). Consider
\begin{displaymath}
|\mathcal H|_d :=
\{h\in\mathbb L^0(\mathbb R;\mathcal M_d(\mathbb R)) :
\forall i,j = 1,\dots,d\textrm{, }
h_{i,j}\in|\mathcal H|\}.
\end{displaymath}
For every $h\in\mathbb L^0(\mathbb R;\mathcal M_d(\mathbb R))$ and $s,t\in\mathbb R$ such that $s < t$ and $h\mathbf 1_{[s,t]}\in|\mathcal H|_d$, the Wiener integral of $h$ with respect to $B$ on $[s,t]$ is the random vector
\begin{displaymath}
\int_{s}^{t}h(u)dB(u) :=
\left(
\sum_{j = 1}^{d}\int_{s}^{t}h_{i,j}(u)dB_j(u)
\right)_{i = 1,\dots,d}.
\end{displaymath}
The following inequality is a straightforward consequence of Memin et al.~\cite[Theorem 1.1]{MMV01} and of basic properties of matrix norms.
%

% Proposition : One-sided isometry for the fractional Wiener integral.

%
\begin{proposition}\label{isometry_fWiener}
There exists a deterministic constant $\mathfrak c_{d,H} > 0$, depending only on $d$ and $H$, such that
\begin{displaymath}
\mathbb E\left(\left\|\int_{s}^{t}h(u)dB(u)\right\|^2\right)
\leqslant
\mathfrak c_{d,H}
\left(\int_{s}^{t}\|h(u)\|_{\normalfont{\textrm{op}}}^{1/H}du\right)^{2H}
\end{displaymath}
for every $s < t$ and $h\in\mathbb L^0(\mathbb R;\mathcal M_d(\mathbb R))$ satisfying $h\mathbf 1_{[s,t]}\in|\mathcal H|_d$.
\end{proposition}
\noindent
Finally, the isometry property (\ref{Wiener_integral_1}) together with the completeness of $\mathbb L^2(\Omega;\mathbb R^d)$ allow to prove the following proposition.
%

% Proposition : Wiener's integral on \mathbb R.

%
\begin{proposition}\label{Wiener_integral_R}
For every $h\in\mathbb L^0(\mathbb R;\mathcal M_d(\mathbb R))$ and $t\in\mathbb R$ such that $h\mathbf 1_{]-\infty,t]}\in|\mathcal H|_d$, there exists a unique $J_{h,B}(t)\in\mathbb L^2(\Omega;\mathbb R^d)$ such that
\begin{displaymath}
\lim_{s\rightarrow\infty}
\mathbb E\left(\left\|J_{h,B}(t) -\int_{-s}^{t}h(u)dB(u)\right\|^2\right) = 0.
\end{displaymath}
The random variable $J_{h,B}(t)$ is the Wiener integral of $h$ with respect to $B$ on $]-\infty,t]$ and it is denoted by
\begin{displaymath}
\int_{-\infty}^{t}h(u)dB(u).
\end{displaymath}
\end{proposition}
%

% Remark : Mishura et al. inequality for Wiener's integral on \mathbb R.

%
\begin{remark}\label{isometry_fWiener_R}
By Propositions \ref{isometry_fWiener} and \ref{Wiener_integral_R}, for any $h\in\mathbb L^0(\mathbb R;\mathcal M_d(\mathbb R))$ and $t\in\mathbb R$ such that $h\mathbf 1_{]-\infty,t]}\in|\mathcal H|_d$,
\begin{eqnarray*}
 \mathbb E\left(\left\|\int_{-\infty}^{t}h(u)dB(u)\right\|^2\right) & = &
 \lim_{s\rightarrow\infty}
 \mathbb E\left(\left\|\int_{-s}^{t}h(u)dB(u)\right\|^2\right)\\
 & \leqslant &
 \mathfrak c_{d,H}
 \left(\int_{-\infty}^{t}\|h(u)\|_{\normalfont{\textrm{op}}}^{1/H}du\right)^{2H}.
\end{eqnarray*}
\end{remark}
\noindent
On the other hand, for $d = 1$ and $T > 0$, consider the reproducing kernel Hilbert space
\begin{displaymath}
\mathfrak H :=
\{h\in\mathbb L^0([0,T]) :
\langle h,h\rangle_{\mathfrak H} <\infty\}
\end{displaymath}
of $B_{|[0,T]}$, where $\langle .,.\rangle_{\mathfrak H}$ is the inner product defined by
\begin{displaymath}
\langle h,\eta\rangle_{\mathfrak H} :=
H(2H - 1)
\int_{0}^{T}\int_{0}^{T}
|t - s|^{2H - 2}h(s)\eta(t)dsdt
\end{displaymath}
for every $h,\eta\in\mathbb L^0([0,T])$. Moreover, let $(\mathbf B(h))_{h\in\mathfrak H}$ be an isonormal Gaussian process associated with the Hilbert space $\mathfrak H$ in the sense of Nualart \cite{NUALART06}, Definition 1.1.1.
%

% Definition : Malliavin derivative.

%
\begin{definition}\label{Malliavin_derivative}
The \emph{Malliavin derivative} of a smooth functional
\begin{displaymath}
F = f(
\mathbf B(h_1),\dots,
\mathbf B(h_n))
\end{displaymath}
where $n\in\mathbb N^*$, $f\in C_{\mathbf p}^{\infty}(\mathbb R^n;\mathbb R)$ and $h_1,\dots,h_n\in\mathfrak H$, is the $\mathfrak H$-valued random variable
\begin{displaymath}
\mathbf DF :=
\sum_{k = 1}^{n}
\partial_k f
(
\mathbf B(h_1),\dots,
\mathbf B(h_n))h_k.
\end{displaymath}
\end{definition}
%

% Proposition : The domain of the Malliavin derivative.

%
\begin{proposition}\label{Malliavin_derivative_domain}
The map $\mathbf D$ is closable from $\mathbb L^2(\Omega,\mathcal A,\mathbb P)$ into $\mathbb L^2(\Omega;\mathfrak H)$. Its domain in $\mathbb L^2(\Omega,\mathcal A,\mathbb P)$, denoted by $\mathbb D^{1,2}$, is the closure of the smooth functionals space for the seminorm $\|.\|_{1,2}$ defined by
\begin{displaymath}
\|F\|_{1,2}^{2} :=
\mathbb E(|F|^2) +
\mathbb E(\|\mathbf DF\|_{\mathfrak H}^{2}) < \infty
\end{displaymath}
for every $F\in\mathbb L^2(\Omega,\mathcal A,\mathbb P)$.
\end{proposition}
\noindent
For a proof, see Nualart \cite[Proposition 1.2.1]{NUALART06}.
%

% Definition : Divergence operator.

%
\begin{definition}\label{divergence_operator}
The adjoint $\delta$ of the Malliavin derivative $\mathbf D$ is the
\emph{divergence operator}. The domain of $\delta$ is denoted by
$\normalfont{\textrm{dom}}(\delta)$, and
$u\in\normalfont{\textrm{dom}}(\delta)$ if and only if there exists a
deterministic constant $\mathfrak c_u > 0$ such that for every
$F\in\mathbb D^{1,2}$, 
\begin{displaymath}
|\mathbb E(\langle\mathbf DF,u\rangle_{\mathfrak H})|
\leqslant
\mathfrak c_u\mathbb E(F^2)^{1/2}.
\end{displaymath}
\noindent
For any process $Y := (Y(s))_{s\in\mathbb R_+}$ and every $t\in [0,T]$, if $Y\mathbf 1_{[0,t]}\in\textrm{dom}(\delta)$, its \emph{Skorokhod integral} with respect to $B$ is defined on $[0,t]$ by
\begin{displaymath}
\int_{0}^{t}Y(s)\delta B(s) :=
\delta(Y\mathbf 1_{[0,t]}).
\end{displaymath}
\end{definition}

% Section : Almost periodic solutions.

%
\section{Almost periodic and periodic solutions to Equation (\ref{main_equation})}\label{section_ap_solutions}
Throughout this section, $A$, $b$ and $\sigma$ fulfill the following assumption.
%

% Assumption : Conditions on b and \sigma.

%
\begin{assumption}\label{conditions_b_sigma}
The functions $S : t\in\mathbb R\mapsto\exp(At)$, $b$ and $\sigma$ satisfy the four following conditions:
\begin{enumerate}
 \item There exist $\mathfrak c_S,\mathfrak m_S > 0$ such that for every $t\in\mathbb R$, $\|S(t)\|_{\normalfont{\textrm{op}}}\leqslant\mathfrak c_Se^{-\mathfrak m_St}$.
 \item There exist $\mathfrak c_b,\mathfrak m_b > 0$ such that for every $t\in\mathbb R$ and $x,y\in\mathbb R^d$,
 \begin{displaymath}
 \|b(t,x) - b(t,y)\|\leqslant\mathfrak c_b\|x - y\|
 \textrm{ and }
 \|b(t,x)\|\leqslant\mathfrak m_b(1 +\|x\|).
 \end{displaymath}
 \item For every $t\in\mathbb R$, $S(t -\cdot)\sigma(\cdot)\mathbf 1_{]-\infty,t]}(\cdot)\in|\mathcal H|_d$.
 \item $b$ (resp. $\sigma$) is almost periodic uniformly with respect to the compact subsets of $\mathbb R^d$ (resp. almost periodic).
\end{enumerate}
\end{assumption}
\noindent
A $d$-dimensional continuous process $X$ is a solution to Equation (\ref{main_equation}) if and only if
\begin{displaymath}
X(t) =
\int_{-\infty}^{t}S(t - s)b(s,X(s))ds +\int_{-\infty}^{t}S(t - s)\sigma(s)dB(s)
\textrm{ $;$ }
\forall t\in\mathbb R.
\end{displaymath}
In order to investigate the question of the existence of almost periodic solutions to Equation (\ref{main_equation}), let $\theta = (\theta_t)_{t\in\mathbb R}$ be the dynamical system on $(\Omega,\mathcal A)$, called Wiener shift, such that
\begin{displaymath}
\theta_t\omega :=
\omega(t +\cdot) -\omega(t)
\end{displaymath}
for every $\omega\in\Omega$ and $t\in\mathbb R$. By Maslowski and Schmalfuss \cite{MS04}, $(\Omega,\mathcal A,\mathbb P,\theta)$ is an ergodic metric dynamical system.
%

% Remark : Invariance of the fBm increments by the translation operator.

%
\begin{remark}\label{invariance_increments_fBm}
For any $t,\tau\in\mathbb R$ and $\omega\in\Omega$,
\begin{displaymath}
B(t +\tau,\theta_{-\tau}\omega) =
B(t,\omega) - B(-\tau,\omega).
\end{displaymath}
Then, for every $s\in\mathbb R$,
\begin{displaymath}
\mathfrak T_{\tau}(B(\cdot + s) - B(\cdot))(t,\omega) = B(t + s,\omega) - B(t,\omega).
\end{displaymath}
\end{remark}
\noindent
For $p\geqslant 1$, 
let $\textrm{AP}^p(\Omega;\mathbb R^d)$ denote the space of continuous, uniformly bounded and $\theta$-almost periodic in $p$-mean processes. Consider also the operator $\Gamma$ defined on $\textrm{AP}^2(\Omega;\mathbb R^d)$ by
\begin{displaymath}
\Gamma X(t) :=
\int_{-\infty}^{t}
S(t - s)b(s,X(s))ds +
\int_{-\infty}^{t}
S(t - s)\sigma(s)dB(s)
\end{displaymath}
for every $X\in\textrm{AP}^2(\Omega;\mathbb R^d)$.
%

% Theorem : Almost periodic solutions.

%
\begin{theorem}\label{almost_periodic_solutions}
Under Assumption \ref{conditions_b_sigma}, $\Gamma$ maps $\normalfont{\textrm{AP}^2}(\Omega;\mathbb R^d)$ into itself. Moreover, if
\begin{displaymath}
\frac{\mathfrak c_S\mathfrak c_b}{\mathfrak m_S} < 1,
\end{displaymath}
then Equation (\ref{main_equation}) has a unique continuous, uniformly bounded and $\theta$-almost periodic in square mean solution.
\end{theorem}
%

% Proof.

%
\begin{proof}
Consider $X\in\textrm{AP}^2(\Omega;\mathbb R^d)$ and $\varepsilon_0 > 0$. The conditions on $S$, $b$ and $\sigma$ together with well known inequalities on Riemman's integral and Propositions \ref{isometry_fWiener} and \ref{Wiener_integral_R} give immediately that $\Gamma X$ is a continuous and uniformly bounded process. It remains to prove, in three steps, that $\Gamma X$ is $\theta$-almost periodic in square mean. A fourth step deals with the existence and uniqueness of the solution to Equation (\ref{main_equation}).
\\
\\
\emph{Step 1.} This is a preliminary step which provides useful controls for Steps 2 and 3. Consider $X\in\textrm{AP}^2(\Omega;\mathbb R^d)$. For any $s\in\mathbb R$, the set $\{X(s +\tau,\theta_{-\tau}.)\textrm{ $;$ }\tau\in\mathbb R\}$ is relatively compact in $\mathbb L^2(\Omega;\mathbb R^d)$ by Proposition \ref{compactness_ap_processes}.(1). Then, $\omega\mapsto X(s +\tau,\theta_{-\tau}\omega)$ is uniformly square integrable with respect to $\tau\in\mathbb R$. By Assumption \ref{conditions_b_sigma}.(2), $\omega\mapsto b(s +\tau,X(s +\tau,\theta_{-\tau}\omega))$ is also uniformly square integrable with respect to $\tau\in\mathbb R$. Therefore, for any $\alpha > 0$, there exists $\eta\in ]0,\alpha\wedge 1[$ such that for any $A\in\mathcal A$,
\begin{equation}\label{almost_periodic_solutions_1}
\forall s,\tau\in\mathbb R\textrm{, }
\mathbb P(A) <\eta
\Longrightarrow
\left\{
\begin{array}{rcl}
 \mathbb E(\|X(s +\tau,\theta_{-\tau}.)\|^2\mathbf 1_A) & < & \alpha\\
 \mathbb E(\|b(s +\tau,X(s +\tau,\theta_{-\tau}.))\|^2\mathbf 1_A) & < & \alpha
\end{array}\right. .
\end{equation}
Moreover, by Proposition \ref{compactness_ap_processes}.(2), there exists a compact subset $K_{\alpha}$ of $\mathbb R^d$ such that
\begin{equation}\label{almost_periodic_solutions_2}
\forall s,\tau\in\mathbb R\textrm{, }
\mathbb P(X(s +\tau,\theta_{-\tau}.)\in K_{\alpha})\geqslant 1 -\eta.
\end{equation}
Finally, by Assumption \ref{conditions_b_sigma}.(4), $b$ (resp. $\sigma$) is uniformly continuous on $\mathbb R\times K_{\alpha}$ (resp. $\mathbb R$) and then, one can choose $\eta$ such that in addition to (\ref{almost_periodic_solutions_1}), for every $s,\tau\in\mathbb R$ satisfying $|\tau - s| <\eta$,
\begin{equation}\label{almost_periodic_solutions_3}
\left\{
\begin{array}{rcl}
 \|\sigma(\tau) -\sigma(s)\|^2 < \alpha
 \textrm{ and }
 \displaystyle{\sup_{x\in K_{\alpha}}\|b(\tau,x) - b(s,x)\|^2} & < & \alpha\\
 \displaystyle{\sup_{u\in\mathbb R}\mathbb E(\|X(u +\tau,\theta_{-\tau}.) - X(u + s,\theta_{-s}.)\|^2)} & < & \alpha
\end{array}\right..
\end{equation}
\emph{Step 2.} Let us establish in this step that for any
$\varepsilon_0 > 0$, the set of $\theta$-$\varepsilon_0$-almost
periods of $\Gamma X$ is relatively dense. By Assumption
\ref{conditions_b_sigma}.(4),
\cite[Corollary 3.4]{RF20} on the almost periodicity in product spaces, and by
\cite[Proposition 3.17]{RF20},  ensuring that a continuous process is $\theta$-almost periodic if and only if its translation is an almost periodic map,
\begin{displaymath}
t\in\mathbb R\longmapsto
(b(t,x),\sigma(t),X(t,.))
\end{displaymath}
is $\theta$-almost periodic uniformly with respect to $x$ in compact subsets of $\mathbb R^d$ (see Definition \ref{almost_periodic_functions}.(4)).
\\
\\
Consider $\varepsilon > 0$ and let $T_{\varepsilon}$ be the relatively dense set of common $\varepsilon$-almost periods of $X$, $b(.,x)$ and $\sigma$ for every $x\in K_{\alpha}$. Let us show that for an appropriate choice of $\varepsilon$ and $\alpha$, the set $T_{\varepsilon}$ is contained in the set of $\varepsilon_0$-almost periods in square mean of $\Gamma X$.
\\
\\
Consider $\tau\in T_{\varepsilon}$ and, without loss of generality, assume that $\tau > 0$. By the definition of $\Gamma X$ together with Remark \ref{invariance_increments_fBm}, for any $t\in\mathbb R$,
\begin{eqnarray*}
 \mathfrak T_{\tau}\Gamma X(t,.) & = &
 \int_{-\infty}^{t +\tau}S(t +\tau - s)b(s,X(s,\theta_{-\tau}.))ds\\
 & & +
 \int_{-\infty}^{t +\tau}S(t +\tau - s)\sigma(s)dB(s,\theta_{-\tau}.)\\
 & = &
 \int_{-\infty}^{t}S(t - s)b(s +\tau,X(s +\tau,\theta_{-\tau}.))ds\\
 & & +
 \int_{-\infty}^{t}S(t - s)\sigma(s +\tau)dB(s,.).
\end{eqnarray*}
So,
\begin{equation}\label{almost_periodic_solutions_4}
\mathbb E(\|\mathfrak T_{\tau}\Gamma X(t,.) -\Gamma X(t,.)\|^2)
\leqslant
3(\mathbb E(I_{\tau}^{1}(t)^2) +\mathbb E(I_{\tau}^{2}(t)^2) +\mathbb E(I_{\tau}^{3}(t)^2))
\end{equation}
where
\begin{eqnarray*}
 I_{\tau}^{1}(t) & := &
 \left\|\int_{-\infty}^{t}S(t - s)(b(s +\tau,X(s +\tau,\theta_{-\tau}.)) - b(s,X(s +\tau,\theta_{-\tau}.)))ds\right\|,\\
 I_{\tau}^{2}(t) & := &
 \left\|\int_{-\infty}^{t}S(t - s)(b(s,X(s +\tau,\theta_{-\tau}.)) - b(s,X(s,.)))ds\right\|
 \textrm{ and}\\
 I_{\tau}^{3}(t) & := &
 \left\|\int_{-\infty}^{t}S(t - s)(\sigma(s +\tau) -\sigma(s))dB(s,.)\right\|.
\end{eqnarray*}
Let us find suitable bounds for $\mathbb E(I_{\tau}^{1}(t)^2)$, $\mathbb E(I_{\tau}^{2}(t)^2)$ and $\mathbb E(I_{\tau}^{3}(t)^2)$.
\begin{enumerate}
 \item For every $s\in\mathbb R$, consider
 \begin{displaymath}
 A_{\alpha}(\tau,s) :=
 \{\omega\in\Omega : X(s +\tau,\theta_{-\tau}\omega)\in K_{\alpha}\}
 \end{displaymath}
 and
 \begin{displaymath}
 b^{\tau}(s,.) :=
 b(s +\tau,X(s +\tau,\theta_{-\tau}.)) -
 b(s,X(s +\tau,\theta_{-\tau}.)).
 \end{displaymath}
 On the one hand, since $\tau$ is an $\varepsilon$-almost period of $b(.,x)$ uniformly with respect to $x\in K_{\alpha}$, for any $s\in\mathbb R$,
 \begin{equation}\label{almost_periodic_solutions_5}
 \|b^{\tau}(s,.)\|\mathbf 1_{A_{\alpha}(\tau,s)}\leqslant\varepsilon.
 \end{equation}
 On the other hand, by (\ref{almost_periodic_solutions_2}),
 \begin{displaymath}
 \mathbb P(A_{\alpha}(\tau,s))\geqslant 1 -\eta\geqslant 1 -\alpha
 \end{displaymath}
 and then by (\ref{almost_periodic_solutions_1}),
 \begin{eqnarray}
  \mathbb E(\|b^{\tau}(s,.)\|^2\mathbf 1_{\Omega\backslash A_{\alpha}(\tau,s)})
  & \leqslant &
  2\mathbb E(\|b(s +\tau,X(s +\tau,\theta_{-\tau}.))\|^2\mathbf 1_{\Omega\backslash A_{\alpha}(\tau,s)})
  \nonumber\\
  & & +
  2\mathbb E(\|b(s,X(s +\tau,\theta_{-\tau}.))\|^2\mathbf 1_{\Omega\backslash A_{\alpha}(\tau,s)})
  \nonumber\\
  \label{almost_periodic_solutions_6}
  & \leqslant &
  4\alpha.
 \end{eqnarray}
 So, by Jensen's inequality, Assumption \ref{conditions_b_sigma}.(1), and Inequalities (\ref{almost_periodic_solutions_5}) and (\ref{almost_periodic_solutions_6}),
 \begin{eqnarray*}
  \mathbb E(I_{\tau}^{1}(t)^2) & \leqslant &
  \mathbb E\left(\left|\int_{-\infty}^{t}\|S(t - s)\|_{\textrm{op}}\|b^{\tau}(s,.)\|ds\right|^2\right)\\
  & \leqslant &
  \mathfrak c_{S}^{2}
  \mathbb E\left(\left|\int_{-\infty}^{t}e^{-\mathfrak m_S(t - s)}\|b^{\tau}(s,.)\|ds\right|^2\right)\\
  & \leqslant &
  \frac{\mathfrak c_{S}^{2}}{\mathfrak m_S}
  \int_{-\infty}^{t}e^{-\mathfrak m_S(t - s)}\mathbb E(\|b^{\tau}(s,.)\|^2)ds
  \leqslant
  \mathfrak c_1(\varepsilon^2 + 4\alpha)
 \end{eqnarray*}
 with
 \begin{displaymath}
 \mathfrak c_1 :=
 \left(\frac{\mathfrak c_S}{\mathfrak m_S}\right)^2.
 \end{displaymath}
 \item By Assumption \ref{conditions_b_sigma}.(1,2,4) and since $\tau$ is a $\theta$-$\varepsilon$-almost period of $X$,
 \begin{eqnarray*}
  \mathbb E(I_{\tau}^{2}(t)^2)
  & \leqslant &
  \mathfrak c_{S}^{2}
  \left(\int_{-\infty}^{t}e^{-\mathfrak m_S(t - s)}ds\right)^2
  \sup_{s\in\mathbb R}
  \mathbb E(\|b(s,X(s +\tau,\theta_{-\tau}.)) - b(s,X(s,.))\|^2)\\
  & \leqslant &
  \mathfrak c_2\varepsilon^2
 \end{eqnarray*}
 with
 \begin{displaymath}
 \mathfrak c_2 :=
 \left(
 \frac{\mathfrak c_S\mathfrak c_b}{\mathfrak m_S}\right)^2.
 \end{displaymath}
 \item By Propositions \ref{isometry_fWiener} and \ref{Wiener_integral_R} together with Assumption \ref{conditions_b_sigma}.(1,3,4),
 \begin{eqnarray*}
  \mathbb E(I_{\tau}^{3}(t)^2)
  & \leqslant &
  \mathfrak c_{d,H}
  \left(\int_{-\infty}^{t}\|S(t - s)(\sigma(s +\tau) -\sigma(s))\|_{\textrm{op}}^{1/H}ds\right)^{2H}\\
  & \leqslant &
  \mathfrak c_{d,H}\mathfrak c_{S}^{2}
  \left(\int_{-\infty}^{t}e^{-\mathfrak m_S(t - s)/H}ds\right)^{2H}
  \varepsilon^2 =
  \mathfrak c_3\varepsilon^2
 \end{eqnarray*}
 with
 \begin{displaymath}
 \mathfrak c_3 :=
 \mathfrak c_{d,H}\mathfrak c_{S}^{2}\left(
 \frac{H}{\mathfrak m_S}\right)^{2H}.
 \end{displaymath}
\end{enumerate}
Therefore, by Inequality (\ref{almost_periodic_solutions_4}),
\begin{displaymath}
\mathbb E(\|\mathfrak T_{\tau}\Gamma X(t,.) -\Gamma X(t,.)\|^2)
\leqslant
3(\mathfrak c_1 +\mathfrak c_2 +\mathfrak c_3)(\varepsilon^2 + 4\alpha).
\end{displaymath}
Since one can take $\varepsilon$ and $\alpha$ such that the right hand side of the previous inequality is lower than $\varepsilon_0$, $T_{\varepsilon}$ is contained in the set of $\theta$-$\varepsilon_0$-almost periods in square mean of $\Gamma X$ as expected. In conclusion, this last set is relatively dense.
\\
\\
\emph{Step 3.} Let us establish in this step that the map
$(t,\tau)\mapsto\mathfrak T_{\tau}\Gamma X(t)$ is continuous for the
distance $d_2$.
Thanks to \cite[Proposition 3.9]{RF20}, it is sufficient to prove the continuity, for the distance $d_2$, of the map $\tau\mapsto\mathfrak T_{\tau}\Gamma X(0)$. Consider $\tau_0,\tau\in\mathbb R$ such that $|\tau -\tau_0| <\eta$ and, without loss of generality, assume that $\tau_0,\tau > 0$. By the definition of $\Gamma X$ together with Remark \ref{invariance_increments_fBm},
\begin{eqnarray*}
 \mathfrak T_{\tau}\Gamma X(0,.) -\mathfrak T_{\tau_0}\Gamma X(0,.)
 & = &
 \int_{-\infty}^{0}S(-s)(b(s +\tau,X(s +\tau,\theta_{-\tau}.))\\
 & & -b(s +\tau_0,X(s +\tau_0,\theta_{-\tau_0}.)))ds\\
 & & +
 \int_{-\infty}^{0}S(-s)(\sigma(s +\tau) -\sigma(s +\tau_0))dB(s,.).
\end{eqnarray*}
So,
\begin{equation}\label{almost_periodic_solutions_7}
\mathbb E(\|\mathfrak T_{\tau}\Gamma X(0,.) -\mathfrak T_{\tau_0}\Gamma X(0,.)\|^2)
\leqslant
3(\mathbb E(|I_{\tau,\tau_0}^{1}|^2) +\mathbb E(|I_{\tau,\tau_0}^{2}|^2) +\mathbb E(|I_{\tau,\tau_0}^{3}|^2))
\end{equation}
where
\begin{eqnarray*}
 I_{\tau,\tau_0}^{1} & := &
 \left\|\int_{-\infty}^{0}S(-s)(b(s +\tau,X(s +\tau_0,\theta_{-\tau_0}.)) - b(s +\tau_0,X(s +\tau_0,\theta_{-\tau_0}.)))ds\right\|,\\
 I_{\tau,\tau_0}^{2} & := &
 \left\|\int_{-\infty}^{0}S(-s)(b(s +\tau,X(s +\tau,\theta_{-\tau}.)) - b(s +\tau,X(s +\tau_0,\theta_{-\tau_0}.)))ds\right\|
 \textrm{ and}\\
 I_{\tau,\tau_0}^{3} & := &
 \left\|\int_{-\infty}^{0}S(-s)(\sigma(s +\tau) -\sigma(s +\tau_0))dB(s,.)\right\|.
\end{eqnarray*}
Let us find suitable bounds for $\mathbb E(|I_{\tau,\tau_0}^{1}|^2)$, $\mathbb E(|I_{\tau,\tau_0}^{2}|^2)$ and $\mathbb E(|I_{\tau,\tau_0}^{3}|^2)$.
\begin{enumerate}
 \item For every $s\in\mathbb R$, consider
 \begin{displaymath}
 b^{\tau,\tau_0}(s,.) :=
 b(s +\tau,X(s +\tau_0,\theta_{-\tau_0}.)) -
 b(s +\tau_0,X(s +\tau_0,\theta_{-\tau_0}.)).
 \end{displaymath}
 On the one hand, by (\ref{almost_periodic_solutions_3}), for any $s\in\mathbb R$,
 \begin{equation}\label{almost_periodic_solutions_8}
 \|b^{\tau,\tau_0}(s,.)\|\mathbf 1_{A_{\alpha}(\tau_0,s)}\leqslant
 \sup_{x\in K_{\alpha}}\|b(s +\tau,x) - b(s +\tau_0,x)\|^2 <\alpha.
 \end{equation}
 On the other hand, by (\ref{almost_periodic_solutions_2}),
 \begin{displaymath}
 \mathbb P(A_{\alpha}(\tau_0,s))\geqslant 1 -\eta\geqslant 1 -\alpha
 \end{displaymath}
 and then by (\ref{almost_periodic_solutions_1}),
 \begin{equation}\label{almost_periodic_solutions_9}
 \mathbb E(\|b^{\tau,\tau_0}(s,.)\|^2\mathbf 1_{\Omega\backslash A_{\alpha}(\tau_0,s)})
 \leqslant
 4\alpha.
 \end{equation}
 So, by Jensen's inequality, Assumption \ref{conditions_b_sigma}.(1), and Inequalities (\ref{almost_periodic_solutions_8}) and (\ref{almost_periodic_solutions_9}),
 \begin{displaymath}
 \mathbb E(|I_{\tau,\tau_0}^{1}|^2)\leqslant
 \frac{\mathfrak c_{S}^{2}}{\mathfrak m_S}
 \int_{-\infty}^{0}e^{\mathfrak m_Ss}\mathbb E(\|b^{\tau,\tau_0}(s,.)\|^2)ds
 \leqslant
 5\mathfrak c_1\alpha.
 \end{displaymath}
 \item By Assumption \ref{conditions_b_sigma}.(1,2) and (\ref{almost_periodic_solutions_3}),
 \begin{eqnarray*}
  \mathbb E(|I_{\tau,\tau_0}^{2}|^2)
  & \leqslant &
  \mathfrak c_{S}^{2}
  \left(\int_{-\infty}^{0}e^{\mathfrak m_Ss}ds\right)^2\\
  & &
  \times
  \sup_{s\in\mathbb R}
  \mathbb E(\|b(s,X(s +\tau,\theta_{-\tau}.)) - b(s,X(s +\tau_0,\theta_{-\tau_0}.))\|^2)\\
  & \leqslant &
  \mathfrak c_2\alpha.
 \end{eqnarray*}
 \item By Propositions \ref{isometry_fWiener} and \ref{Wiener_integral_R}, Assumption \ref{conditions_b_sigma}.(1,3) and (\ref{almost_periodic_solutions_3}),
 \begin{eqnarray*}
  \mathbb E(|I_{\tau,\tau_0}^{3}|^2)
  & \leqslant &
  \mathfrak c_{d,H}
  \left(\int_{-\infty}^{0}\|S(-s)(\sigma(s +\tau) -\sigma(s +\tau_0))\|_{\textrm{op}}^{1/H}ds\right)^{2H}\\
  & \leqslant &
  \mathfrak c_{d,H}\mathfrak c_{S}^{2}
  \left(\int_{-\infty}^{0}e^{\mathfrak m_Ss/H}ds\right)^{2H}
  \alpha =
  \mathfrak c_3\alpha.
 \end{eqnarray*}
\end{enumerate}
Therefore, by Inequality (\ref{almost_periodic_solutions_7}),
\begin{displaymath}
\mathbb E(\|\mathfrak T_{\tau}\Gamma X(0,.) -\mathfrak T_{\tau_0}\Gamma X(0,.)\|^2)
\leqslant
15(\mathfrak c_1 +\mathfrak c_2 +\mathfrak c_3)\alpha.
\end{displaymath}
Since $\alpha$ has been chosen arbitrarily close to $0$, the map $\tau\mapsto\mathfrak T_{\tau}\Gamma X(0)$ is continuous at time $\tau_0$ for the distance $d_2$.
\\
\\
\emph{Step 4.} For every $X,X^*\in\textrm{AP}^2(\Omega;\mathbb R^d)$ and $t\in\mathbb R$, by Jensen's inequality and Assumption \ref{conditions_b_sigma}.(1,2),
\begin{eqnarray*}
 \mathbb E(\|\Gamma X(t) -\Gamma X^*(t)\|^2)
 & \leqslant &
 \mathfrak c_{S}^{2}\mathbb E\left(
 \left|\int_{-\infty}^{t}e^{-\mathfrak m_S(t - s)}
 \|b(s,X(s)) - b(s,X^*(s))\|ds\right|^2\right)\\
 & \leqslant &
 \frac{\mathfrak c_{S}^{2}}{\mathfrak m_S}
 \int_{-\infty}^{t}e^{-\mathfrak m_S(t - s)}\mathbb E(\|b(s,X(s)) - b(s,X^*(s))\|^2)ds\\
 & \leqslant &
 \frac{\mathfrak c_{S}^{2}\mathfrak c_{b}^{2}}{\mathfrak m_S}
 \left(\int_{-\infty}^{t}e^{-\mathfrak m_S(t - s)}ds\right)\\
 & &
 \times\sup_{s\in\mathbb R}\mathbb E(\|X(s) - X^*(s)\|^2) =
 \mathfrak c_2\sup_{s\in\mathbb R}\mathbb E(\|X(s) - X^*(s)\|^2).
\end{eqnarray*}
Since $\mathfrak c_2 < 1$, $\Gamma$ has a unique fixed point by Picard's theorem.
\end{proof}
%
%%%%%%%%%%%%%%%%%%%%%%%%%%%%%%%%%%%%%%%%%%%%%%%%%%
\begin{remark}[Square mean almost periodicity and
  fractional Ornstein-Uhlenbeck process]\label{rem:counterexample} 
The simplest case of Equation \eqref{main_equation}, with $d=1$, $b=0$
and where $\sigma$ is a constant (fractional Ornstein-Uhlenbeck process)
shows that ``plain'' almost periodicity in square mean
(that is, with $\theta_t=\mathrm{Id}_\Omega$ for all $t\in\mathbb{R}$)
is inapplicable for equations driven by fractional Brownian motion.
Indeed, by Cheridito et al.~\cite[Theorem 2.3]{cheridito}, the autocovariance
function of the fractional Ornstein-Uhlenbeck process decays to
0. However, this process is stationary, thus it has a constant
variance. By 
\cite[Lemma 2.3]{counterexamples}
this shows that no nontrivial fractional Ornstein-Uhlenbeck process is 
almost periodic in square mean. However, Theorem
\ref{almost_periodic_solutions} shows that it is always $\theta$-almost
periodic in square mean. 
\end{remark}

\noindent
Now, $S$, $b$ and $\sigma$ fulfill the following assumption.
%

% Assumption : Conditions on b and \sigma (periodic case).

%
\begin{assumption}\label{conditions_b_sigma+}
The functions $S : t\in\mathbb R\mapsto\exp(At)$, $b$ and $\sigma$ satisfy the four following conditions:
\begin{enumerate}
 \item There exist $\mathfrak c_S,\mathfrak m_S > 0$ such that for every $t\in\mathbb R$, $\|S(t)\|_{\normalfont{\textrm{op}}}\leqslant\mathfrak c_Se^{-\mathfrak m_St}$.
 \item There exist $\mathfrak c_b,\mathfrak m_b > 0$ such that for every $t\in\mathbb R$ and $x,y\in\mathbb R^d$,
 \begin{displaymath}
 \|b(t,x) - b(t,y)\|\leqslant\mathfrak c_b\|x - y\|
 \textrm{ and }
 \|b(t,x)\|\leqslant\mathfrak m_b(1 +\|x\|).
 \end{displaymath}
 \item For every $t\in\mathbb R$, $S(t -\cdot)\sigma(\cdot)\mathbf 1_{]-\infty,t]}(\cdot)\in|\mathcal H|_d$.
 \item There exists $\tau > 0$ such that $b(.,x)$ (resp. $\sigma$) is $\tau$-periodic for every $x\in\mathbb R^d$ (resp. $\tau$-periodic).
\end{enumerate}
\end{assumption}
\noindent
Assumption \ref{conditions_b_sigma+} is stronger than Assumption \ref{conditions_b_sigma} because its fourth item deals with periodicity of the vector field of Equation (\ref{main_equation}) instead of almost periodicity.
\\
\\
Under Assumption \ref{conditions_b_sigma+}, the proof of the following
proposition remains the same as that of Theorem \ref{almost_periodic_solutions} by taking $\varepsilon_0 = 0$.
%

% Proposition : Periodic solutions.

%
\begin{proposition}\label{periodic_solutions}
Under Assumption \ref{conditions_b_sigma+}, if
\begin{displaymath}
\frac{\mathfrak c_S\mathfrak c_b}{\mathfrak m_S} < 1,
\end{displaymath}
then Equation (\ref{main_equation}) has a unique continuous, uniformly bounded and $\theta$-$\tau$-periodic solution.
\end{proposition}
%

% Section : Consistency of an estimator of the parameter \vartheta in Equation (\ref{main_equation_estimation}).

%
\section{Consistency of an estimator of the parameter $\vartheta$ in Equation (\ref{main_equation_estimation})}\label{section_parameter_estimation}
Throughout this section, the parameter $\vartheta$ involved in
Equation (\ref{main_equation_estimation}) belongs to
$[\underline\vartheta,\infty[$ with $\underline\vartheta >
0$. Moreover, the function $b_0$ fulfills the following assumption. 
\begin{assumption}\label{assumption_b_0}
The functions $b_0(t,.)$, $t\in\mathbb R_+$ belong to $C^1(\mathbb R;\mathbb R)\backslash\{\textrm{Id}_{\mathbb R}\}$ and there exists $\underline{\mathfrak m}_{b_0},\overline{\mathfrak m}_{b_0}\in ]0,1[$ such that
\begin{displaymath}
-\overline{\mathfrak m}_{b_0}
\leqslant
\partial_2b_0(t,x)\leqslant 1 -\underline{\mathfrak m}_{b_0}
\end{displaymath}
for every $(t,x)\in\mathbb R_+\times\mathbb R$.
\end{assumption}
\noindent
For instance, assume that
\begin{displaymath}
b_0(t,x) := u(t)v(x)
\textrm{ $;$ }
\forall (t,x)\in\mathbb R^2,
\end{displaymath}
where $u :\mathbb R\rightarrow\mathbb R$ is a continuous almost periodic function and $v\in C^1(\mathbb R;\mathbb R)$. If
\begin{displaymath}
u(t)v(x)\not= x
\textrm{ $;$ }
\forall (t,x)\in\mathbb R_+\times\mathbb R
\end{displaymath}
and
\begin{displaymath}
\exists\varepsilon > 0 :
\{u(t)v'(x)\textrm{ $;$ }(t,x)\in\mathbb R_+\times\mathbb R\}\subset ]-1 +\varepsilon,1 -\varepsilon[,
\end{displaymath}
then $b = -\vartheta b_0$ (resp. $b_0$) fulfils Assumption \ref{conditions_b_sigma}.(2,4) (resp. \ref{assumption_b_0}).
\\
\\
\textbf{Practical examples:}
\begin{enumerate}
 \item If
 \begin{displaymath}
 b_0(t,x) :=
 \frac{1}{4}(\cos(t) +\sin(\sqrt 2\cdot t))x
 \textrm{ $;$ }
 \forall (t,x)\in\mathbb R^2,
 \end{displaymath}
 then $b = -\vartheta b_0$ (resp. $b_0$) fulfils Assumption \ref{conditions_b_sigma}.(2,4) (resp. \ref{assumption_b_0}).
 \item If
 \begin{displaymath}
 b_0(t,x) :=
 \frac{1}{4}(\sin(t) +\cos(\sqrt 2\cdot t))\arctan(x)
 \textrm{ $;$ }
 \forall (t,x)\in\mathbb R^2,
 \end{displaymath}
 then $b = -\vartheta b_0$ (resp. $b_0$) fulfils Assumption \ref{conditions_b_sigma}.(2,4) (resp. \ref{assumption_b_0}).
 \item If
 \begin{eqnarray*}
  b_0(t,x) & := &
  \frac{1}{8}(\cos(t) +\sin(\sqrt 2\cdot t))x\\
  & &
  +\frac{1}{8}(\sin(t) +\cos(\sqrt 2\cdot t))\arctan(x)
  \textrm{ $;$ }
  \forall (t,x)\in\mathbb R^2,
 \end{eqnarray*}
 then $b = -\vartheta b_0$ (resp. $b_0$) fulfils Assumption \ref{conditions_b_sigma}.(2,4) (resp. \ref{assumption_b_0}).
 \item If
 \begin{displaymath}
 b_0(t,x) :=
 \frac{1}{2}\cos(t)x
 \textrm{ $;$ }
 \forall (t,x)\in\mathbb R^2,
 \end{displaymath}
 then $b = -\vartheta b_0$ (resp. $b_0$) fulfils Assumption \ref{conditions_b_sigma+}.(2,4) (resp. \ref{assumption_b_0}).
\end{enumerate}
Note that under Assumption \ref{assumption_b_0}, $\mathfrak c_b =\vartheta [(1 -\underline{\mathfrak m}_{b_0})\vee\overline{\mathfrak m}_{b_0}]$. So, under Assumption \ref{conditions_b_sigma} (resp. \ref{conditions_b_sigma+}), since $\mathfrak c_S = 1$ and $\mathfrak m_S =\vartheta$,
\begin{displaymath}
\frac{\mathfrak c_S\mathfrak c_b}{\mathfrak m_S} =
(1 -\underline{\mathfrak m}_{b_0})\vee\overline{\mathfrak m}_{b_0} < 1,
\end{displaymath}
and then Equation (\ref{main_equation_estimation}) has a unique almost periodic (resp. periodic) solution by Theorem \ref{almost_periodic_solutions} (resp. Proposition \ref{periodic_solutions}).
%

% Remark : About Assumption assumption_b_0.

%
\begin{remark}\label{about_assumption_b_0}
Let us give some details about Assumption \ref{assumption_b_0}. On the
one hand, the assumption
\begin{displaymath}
\partial_2b_0(t,x)\leqslant 1 -\underline{\mathfrak m}_{b_0}
\textrm{ $;$ }
\forall (t,x)\in\mathbb R_+\times\mathbb R
\end{displaymath}
is crucial in order to prove Lemma \ref{control_divergence_integral}. This condition on $\partial_2b_0$ plays the same role for Equation (\ref{main_equation_estimation}) than Hu et al. \cite{HNZ18}, Hypothesis 1.1 for autonomous equations. On the other hand, as seen above, to assume that $\underline{\mathfrak m}_{b_0}$ and $\overline{\mathfrak m}_{b_0}$ belong to $]0,1[$ allows to show that
\begin{displaymath}
\frac{\mathfrak c_S\mathfrak c_b}{\mathfrak m_S} =
(1 -\underline{\mathfrak m}_{b_0})\vee\overline{\mathfrak m}_{b_0} < 1,
\end{displaymath}
and then to apply Theorem \ref{almost_periodic_solutions} and Proposition \ref{periodic_solutions}. This last condition can be improved by replacing the drift function $(t,x)\mapsto -\vartheta (x - b_0(t,x))$ by $(t,x)\mapsto -\vartheta (m_0x - b_0(t,x))$ with a known $m_0\geqslant 1$. For the sake of readability, $m_0 = 1$ in this paper, but the case $m_0 > 1$ doesn't generate additional difficulties and is left to the reader.
\end{remark}
\noindent
Under Assumptions \ref{conditions_b_sigma} or \ref{conditions_b_sigma+}, and Assumption \ref{assumption_b_0}, the purpose of this section is to establish the consistency of the least-square type estimator
\begin{displaymath}
\widehat\vartheta_T :=
-\frac{\displaystyle{\int_{0}^{T}(X(s) - b_0(s,X(s)))\delta X(s)}}{\displaystyle{\int_{0}^{T}(X(s) - b_0(s,X(s)))^2ds}}
\textrm{ $;$ }T > 0
\end{displaymath}
of $\vartheta$, where the Skorokhod integral with respect to the solution $X$ to Equation (\ref{main_equation_estimation}) is defined by
\begin{displaymath}
\int_{0}^{t}
Y(s)\delta X(s) :=
-\vartheta\int_{0}^{t}Y(s)(X(s) - b_0(s,X(s)))ds +
\int_{0}^{t}Y(s)\sigma(s)\delta B(s)
\end{displaymath}
for any continuous process $Y$ and every $t\in [0,T]$ such that $Y\sigma\mathbf 1_{[0,t]}\in\textrm{dom}(\delta)$.
%

% Remark : About the computability of the estimator of \theta.

%
\begin{remark}\label{computability_estimator}
Note that except in the case $H = 1/2$ because the Skorokhod integral coincides with It\^o's integral on its domain, the estimator $\widehat{\vartheta}_T$ is difficult to compute. However, in some recent papers (see Comte and Marie \cite{CM19,CM20}), the authors proposed a procedure to compute Skorokhod's integral based estimators requiring an observed path of the solution for two close but different values of the initial condition. Clearly, such a requirement is not possible in any context, but the authors had in mind the pharmacokinetics application field and explained why it is meaningful in this context. Since Equation (\ref{main_equation}) is defined on $\mathbb R$, the procedure of \cite{CM19,CM20} cannot be transposed directly to our estimator $\widehat{\vartheta}_T$, but an extension will be investigated in a forthcoming work.
\end{remark}
\noindent
Let $C_{\mathbf b}^{1}(\mathbb R_+\times\mathbb R,\mathbb R)$ be the subspace of $C^0(\mathbb R_+\times\mathbb R,\mathbb R)$ such that $\varphi\in C_{\mathbf b}^{1}(\mathbb R_+\times\mathbb R,\mathbb R)$ if and only if, for every $t\in\mathbb R_+$, $\varphi(t,.)$ belongs to $C^1(\mathbb R;\mathbb R)$ and $\partial_2\varphi$ is bounded.
\\
\\
The following lemma is similar to Hu et al.~\cite[Proposition 4.4]{HNZ18}.
%

% Lemma : Control of the divergence integral.

%
\begin{lemma}\label{control_divergence_integral}
Under Assumptions \ref{conditions_b_sigma} and \ref{assumption_b_0}, there exists a deterministic constant $\mathfrak c_{H,\sigma,\underline\vartheta} > 0$, only depending on $H$, $\|\sigma\|_{\infty}$ and $\underline\vartheta$, such that for every $\varphi\in C_{\mathbf b}^{1}(\mathbb R_+\times\mathbb R,\mathbb R)$,
\begin{eqnarray*}
 \mathbb E\left(
 \left|\int_{0}^{T}\varphi(s,X(s))\delta B(s)\right|^2\right)
 & \leqslant &
 \mathfrak c_{H,\sigma,\underline\vartheta}\left[\left(\int_{0}^{T}\mathbb E(|\varphi(s,X(s))|^{1/H})ds\right)^{2H}\right.\\
 & &
 +\left.
 \left(\int_{0}^{T}\mathbb E(|\partial_2\varphi(s,X(s))|^2)^{1/(2H)}ds\right)^{2H}\right] <\infty.
\end{eqnarray*}
\end{lemma}
%

% Proof.

%
\begin{proof}
On the one hand, for any $s,t\in [0,T]$, by the chain rule for Malliavin's derivative,
\begin{displaymath}
\mathbf D_sX(t) =
\sigma(s)\mathbf 1_{[0,t]}(s)
-\vartheta\int_{0}^{t}
(1 -\partial_2b_0(u,X(u)))\mathbf D_sX(u)du.
\end{displaymath}
Then,
\begin{displaymath}
\mathbf D_sX(t) =
\sigma(s)\mathbf 1_{[0,t]}(s)\exp\left(-\vartheta\int_{s}^{t}(1 -\partial_2b_0(u,X(u)))du\right)
\end{displaymath}
and, by Assumption \ref{assumption_b_0},
\begin{equation}\label{control_divergence_integral_1}
|\mathbf D_sX(t)|
\leqslant
\|\sigma\|_{\infty}
\mathbf 1_{[0,t]}(s)
e^{-\underline\vartheta\cdot\underline{\mathfrak m}_{b_0}(t - s)}.
\end{equation}
On the other hand, by Hu et al.~\cite[Theorem 3.6.(2)]{HNZ18}, there exists a deterministic constant $\mathfrak c_H > 0$, depending only on $H$, such that for any $\varphi\in C_{\mathbf b}^{1}(\mathbb R_+\times\mathbb R,\mathbb R)$,
\begin{eqnarray}
 \mathbb E\left(
 \left|\int_{0}^{T}\varphi(u,X(u))\delta B(u)\right|^2\right)
 & \leqslant &
 \mathfrak c_H\left[\left(\int_{0}^{T}\mathbb E(|\varphi(u,X(u))|^{1/H})du\right)^{2H}\right.
 \nonumber\\
 \label{control_divergence_integral_2}
 & &
 +\left.
 \mathbb E\left(\int_{0}^{T}
 \int_{0}^{u}|\mathbf D_v[\varphi(u,X(u))]|^{1/H}dvdu\right)^{2H}\right].
\end{eqnarray}
As in the proof of Hu et al.~\cite[Proposition 4.4]{HNZ18}, Inequalities (\ref{control_divergence_integral_1}) and (\ref{control_divergence_integral_2}) allow to conclude.
\end{proof}
\noindent
Now, let us establish the consistency of the estimator $\widehat\vartheta_T$ under Assumption \ref{conditions_b_sigma+} (periodic case), and then under Assumption \ref{conditions_b_sigma} (almost periodic case). Lemma \ref{mean_value_periodic} is a little bit stronger than Lemma \ref{mean_value_almost_periodic}, and to investigate the periodic case first helps to understand the almost periodic one.
%

% Subsection : Consistency of \widehat\vartheta_T: periodic case.

%
\subsection{Consistency of $\widehat\vartheta_T$: periodic case}
For every $\tau > 0$, consider
\begin{displaymath}
\textrm{Per}_{\tau}(\Omega;\mathbb R) :=
\{Y\in\textrm{AP}^1(\Omega;\mathbb R) :
Y\textrm{ is $\theta$-$\tau$-periodic}\}.
\end{displaymath}
The following lemma is a mean value theorem for the elements of $\textrm{Per}_{\tau}(\Omega;\mathbb R)$.
%

% Lemma : Ergodic theorem (periodic).

%
\begin{lemma}\label{mean_value_periodic}
For every $\tau > 0$ and $Y\in\normalfont{\textrm{Per}}_{\tau}(\Omega;\mathbb R)$,
\begin{displaymath}
\frac{1}{t}\int_{0}^{t}Y(s)ds
\xrightarrow[t\rightarrow\infty]{\normalfont{\textrm{a.s.}}/\mathbb L^2}
\frac{1}{\tau}\int_{0}^{\tau}\mathbb E(Y(s))ds.
\end{displaymath}
\end{lemma}
%

% Proof.

%
\begin{proof}
Consider $\tau > 0$ and $Y\in\textrm{Per}_{\tau}(\Omega;\mathbb R)$. Without loss of generality, by taking $t = n\tau$ with $n\in\mathbb N^*$,
\begin{eqnarray*}
 \frac{1}{t}\int_{0}^{t}Y(s,.)ds  & = &
 \frac{1}{t}\sum_{k = 0}^{n - 1}\int_{k\tau}^{(k + 1)\tau}Y(s,.)ds\\
 & = &
 \frac{1}{n\tau}\sum_{k = 0}^{n - 1}\int_{0}^{\tau}Y(s + k\tau,\theta_{-k\tau}(\theta_{k\tau}.))ds
 =\frac{1}{\tau}\int_{0}^{\tau}M_{n}^{\tau}(s,.)ds
\end{eqnarray*}
where
\begin{displaymath}
M_{n}^{\tau}(s,.) :=
\frac{1}{n}\sum_{k = 0}^{n - 1}Y(s,\theta_{k\tau}.)
\textrm{ $;$ }
\forall s\in\mathbb R_+.
\end{displaymath}
Since $(\Omega,\mathcal A,\mathbb P,\theta)$ is an ergodic metric dynamical system (see Maslowski and Schmalfuss \cite{MS04}), by Birkhoff's ergodic theorem,
\begin{displaymath}
M_{n}^{\tau}(s)
\xrightarrow[n\rightarrow\infty]{\textrm{a.s.}/\mathbb L^2}
\mathbb E(Y(s))
\textrm{ $;$ }
\forall s\in\mathbb R_+.
\end{displaymath}
Moreover, since it belongs to $\textrm{Per}_{\tau}(\Omega;\mathbb R)$, the process $Y$ is bounded. So, by Lebesgue's theorem,
\begin{displaymath}
\frac{1}{\tau}\int_{0}^{\tau}M_{n}^{\tau}(s)ds
\xrightarrow[n\rightarrow\infty]{\textrm{a.s.}/\mathbb L^2}
\frac{1}{\tau}\int_{0}^{\tau}
\mathbb E(Y(s))ds.
\end{displaymath}
This concludes the proof.
\end{proof}
\noindent
Note that the preceding lemma obviously holds if $Y$ is a finite sum
of $\theta$-periodic processes.
%

% Proposition : Consistency of the least-square type estimator (periodic solution).

%
\begin{proposition}\label{consistency_LS_estimator_periodic}
Under Assumptions \ref{conditions_b_sigma+} and \ref{assumption_b_0}, $\widehat\vartheta_T$ is a consistent estimator of $\vartheta$.
\end{proposition}
%

% Proof.

%
\begin{proof}
First of all, note that $\widehat\vartheta_T =\vartheta - U_T/V_T$, where
\begin{displaymath}
U_T :=
\frac{1}{T}
\int_{0}^{T}(X(s) - b_0(s,X(s)))\sigma(s)\delta B(s)
\end{displaymath}
and
\begin{displaymath}
V_T :=
\frac{1}{T}
\int_{0}^{T}(X(s) - b_0(s,X(s)))^2ds.
\end{displaymath}
On the one hand, let us show that $\mathbb E(U_{T}^{2})\rightarrow 0$ as $T\rightarrow\infty$. By Lemma \ref{control_divergence_integral}, and since $b_0$, $\partial_2b_0$, $\sigma$ and $s\in\mathbb R\mapsto\mathbb E(X(s)^2)$ are bounded under Assumptions \ref{conditions_b_sigma+} and \ref{assumption_b_0},
\begin{eqnarray*}
 \mathbb E(U_{T}^{2})
 & \leqslant &
 \frac{\mathfrak c_{H,\sigma,\underline\vartheta}}{T^2}\left[\left(\int_{0}^{T}\mathbb E(|(X(s) - b_0(s,X(s)))\sigma(s)|^{1/H})ds\right)^{2H}\right.\\
 & &
 +\left.
 \left(\int_{0}^{T}\mathbb E(|(1 -\partial_2b_0(s,X(s)))\sigma(s)|^2)^{1/(2H)}ds\right)^{2H}\right]\\
 & \leqslant &
 \frac{\mathfrak c_{H,\sigma,\underline\vartheta}\|\sigma\|_{\infty}^{2}}{T^{2 - 2H}}
 \left[(1 +\mathfrak m_b)^2\sup_{s\in\mathbb R}\mathbb E((1 + |X(s)|)^2) +
 (1 + (1 -\underline{\mathfrak m}_{b_0})\vee\overline{\mathfrak m}_{b_0})^2
 \right]\\
 & &
 \quad\quad
 \xrightarrow[T\rightarrow\infty]{} 0.
\end{eqnarray*}
On the other hand, by Lemma \ref{mean_value_periodic},
\begin{displaymath}
V_T\xrightarrow[T\rightarrow\infty]{\mathbb L^2}
\frac{1}{\tau}\int_{0}^{\tau}\mathbb E((X(s) - b_0(s,X(s)))^2)ds > 0.
\end{displaymath}
Therefore, by Slutsky's lemma,
\begin{displaymath}
\widehat\vartheta_T
\xrightarrow[T\rightarrow\infty]{\mathbb P}\vartheta.
\end{displaymath}
\end{proof}
%

% Subsection : Consistency of \widehat\vartheta_T: almost periodic case.

%
\subsection{Consistency of $\widehat\vartheta_T$: almost periodic case}
The following lemma is a mean value theorem for the elements of
$\textrm{AP}^1(\Omega;\mathbb R)$. Since this lemma provides a
convergence result in $\mathbb L^1(\Omega;\mathbb R)$, for the
processes of $\textrm{Per}_{\tau}(\Omega;\mathbb R)$ with $\tau > 0$,
its conclusion is slightly weaker than that of
Lemma \ref{mean_value_periodic} which provides the same convergence
result in
$\mathbb L^2(\Omega;\mathbb R)$.
%

% Lemma : Ergodic theorem (almost periodic).

%
\begin{lemma}\label{mean_value_almost_periodic}
For every $Y\in\normalfont{\textrm{AP}^{1}}(\Omega;\mathbb R)$, the mean value $\mathcal M(m_Y)$ of its mean function $m_Y : s\mapsto\mathbb E(Y(s))$ exists and
\begin{displaymath}
\frac{1}{t}\int_{0}^{t}Y(s,.)ds
\xrightarrow[t\rightarrow\infty]{\normalfont{\mathbb L^1}}
\mathcal{M}(m_Y).
\end{displaymath}
\end{lemma}
%

% Proof.

%
\begin{proof}
Let $Y\in\normalfont{\textrm{AP}^{1}}(\Omega;\mathbb R)$.
Since $m_Y$ is an almost periodic function, 
its mean value $\mathcal{M}(m_Y)$ exists by Proposition \ref{mean_value_ap_functions}. Let $\varepsilon>0$, and let $T_{\varepsilon/3}$ denote the set of
$\theta$-$\frac{\varepsilon}{3}$-periods of $Y$.
Since $T_{\varepsilon/3}$ is relatively dense, we can
choose $\tau\in T_{\varepsilon/3}$ such that 
\begin{equation}\label{eq:i1}
  %(t\geqslant \tau)\Rightarrow
  \left|\frac{1}{\tau}\int_0^\tau\mathbb{E}(Y(s))ds
  - \mathcal{M}(m_Y)\right|\leqslant\frac{\varepsilon}{3}.
\end{equation}
Let us denote, for $n\in\mathbb{N}^*$ and $s\in\mathbb R_+$,
\[
M_{n}^{\tau}(s,.)=\frac{1}{n}\sum_{k=0}^{n-1}Y(s,\theta_{k\tau}.). 
\]
Since $(\Omega,\mathcal{A},\mathbb{P},\theta)$ is
an ergodic metric dynamical system (see Maslowski and Schmalfuss \cite{MS04}),
we deduce by Birkhoff's theorem
\begin{displaymath}
M_{n}^{\tau}(s)
\xrightarrow[n\rightarrow\infty]{\textrm{a.s.}/\mathbb L^1}
\mathbb{E}(Y(s))
\textrm{ $;$ }
\forall s\in\mathbb R_+.
\end{displaymath}
Using the uniform continuity on $[0,\tau]$ of $s\mapsto M_n(s)$ in
$\mathbb{L}^1$, we deduce 
\begin{equation*}
  \left|
    \frac{1}{\tau}\int_0^\tau M_n(s)ds
    - \frac{1}{\tau}\int_0^\tau \mathbb{E}(Y(s))ds
    \right|
 \xrightarrow[n\rightarrow\infty]{\mathbb L^1} 0. 
\end{equation*}
In particular, there exists $N\in\mathbb N$ large enough such that
\begin{equation}\label{eq:i2}
  \mathbb{E}\left(\left|
    \frac{1}{\tau}\int_0^\tau M_n(s)ds
    - \frac{1}{\tau}\int_0^\tau \mathbb{E}(Y(s))ds
    \right|\right)\leqslant\frac{\varepsilon}{3};\quad \forall n\geqslant N.
\end{equation}
On the other hand, we have 
\begin{multline}\label{eq:i3}
 \mathbb{E}\left(\left|
 \frac{1}{n\tau}\int_{0}^{n\tau}Y(s,.)ds-\frac{1}{\tau}\int_0^\tau M_n(s,.)ds
 \right|\right)\\
\begin{aligned}
=& \frac{1}{n\tau} \mathbb{E}\left(\left|
\sum_{k=0}^{n-1}\int_{k\tau}^{(k+1)\tau}Y(s)ds
-\sum_{k=0}^{n-1}\int_0^\tau Y(s,\theta_{k\tau}.)
\right|\right)\\
\leqslant &\frac{1}{n\tau} \sum_{k=0}^{n-1}\int_0^\tau
\mathbb{E}(|Y(s+k\tau,.)-Y(s,\theta_{k\tau}.)|)ds
\leqslant \frac{\varepsilon}{3}.
\end{aligned}
\end{multline}
From \eqref{eq:i1}-\eqref{eq:i2}-\eqref{eq:i3}, we deduce that
\begin{equation*}
   \mathbb{E}\left(\left|
     \frac{1}{n\tau}\int_{0}^{n\tau}Y(s)ds-\mathcal{M}(m_Y)
                  \right|\right)\leqslant\varepsilon; \quad\forall n\geqslant N. 
\end{equation*}
To conclude the proof, we only need to notice that,         
for $t=n\tau+r$, with $0\leqslant r<\tau$, we have,
since $s\mapsto \mathbb{E}(Y(s,.))$ is bounded,
\begin{multline*}
   \mathbb{E}\left(\left|
     \frac{1}{t}\int_{0}^{t}Y(s)ds-\frac{1}{n\tau}\int_{0}^{n\tau}Y(s)ds
              \right|\right)\\
            \leqslant\left(1 -\frac{n\tau}{n\tau+r}\right)
            \mathbb{E}\left(\left|
              \frac{1}{n\tau}\int_{0}^{n\tau}Y(s)ds
              \right|\right)
              +\frac{1}{n\tau+r}\int_{0}^{\tau}\mathbb E(|Y(n\tau + s)|)ds\\
      \xrightarrow[n\rightarrow\infty]{}0\text{ uniformly with respect
        to $r$}.        
\end{multline*}
\end{proof}
%

% Proposition : Consistency of the least-square type estimator (almost periodic solution).

%
\begin{proposition}\label{consistency_LS_estimator_ap}
Under Assumptions \ref{conditions_b_sigma} and \ref{assumption_b_0}, $\widehat\vartheta_T$ is a consistent estimator of $\vartheta$.
\end{proposition}
%

% Proof.

%
\begin{proof}
As established in the proof of Proposition \ref{consistency_LS_estimator_periodic}, $\widehat\vartheta_T =\vartheta - U_T/V_T$ where
\begin{displaymath}
U_T =
\frac{1}{T}
\int_{0}^{T}(X(s) - b_0(s,X(s)))\sigma(s)\delta B(s)
\xrightarrow[T\rightarrow\infty]{\mathbb L^2} 0,
\end{displaymath}
and
\begin{displaymath}
V_T =
\frac{1}{T}
\int_{0}^{T}Y(s)ds
\end{displaymath}
with
\begin{displaymath}
Y(s) := (X(s) - b_0(s,X(s)))^2
\textrm{ $;$ }
\forall s\in\mathbb R.
\end{displaymath}
Since $X\in\textrm{AP}^2(\Omega;\mathbb R)$ by Theorem
\ref{almost_periodic_solutions} and the functions $b_0(.,x)$,
$x\in\mathbb R$ are almost periodic, $Y\in\textrm{AP}^1(\Omega;\mathbb
R)$ by Bochner's double sequence criterion
(see \cite[Theorem 3.12]{RF20}).
Then, by Lemma \ref{mean_value_almost_periodic}, 
\begin{displaymath}
V_T\xrightarrow[T\rightarrow\infty]{\mathbb L^1}
\mathcal M(\mu_{Y}^{2})
\end{displaymath}
where $\mu_Y$ is the square root of the mean function $m_Y$ of
$Y$. Since $m_Y$ is almost periodic, $\mu_Y$ is also this is also the
case,
by Bochner's double sequence criterion. Then, by Parseval's equality (see Proposition \ref{Parseval_equality_ap_functions}), for any sequence $(\lambda_n)_{n\in\mathbb N^*}$ of elements of $\mathbb S(\mu_Y)$,
\begin{displaymath}
\mathcal M(\mu_{Y}^{2}) =
\sum_{n = 1}^{\infty}
|\mathcal M(\mu_Ye^{i\lambda_n.})|^2.
\end{displaymath}
So, $\mathcal M(\mu_{Y}^{2}) > 0$ because if $\mathcal M(\mu_{Y}^{2}) = 0$, then $X(.) = b_0(.,X(.))$ almost everywhere. Therefore, by Slutsky's lemma,
\begin{displaymath}
\widehat\vartheta_T
\xrightarrow[T\rightarrow\infty]{\mathbb P}\vartheta.
\end{displaymath}
\end{proof}

\noindent
\textbf{Acknowledgments.}
We thank the reviewer for his valuable comments which helped improve
the readability of this article. 
This work was funded by RFBR and CNRS, project number PRC2767. 
% References.

%

%
\end{document}